\documentclass[11pt,a4paper]{article}
\usepackage[top=3cm, bottom=3cm, left=3cm, right=3cm]{geometry}
\usepackage[latin1]{inputenc}
\usepackage{amsmath}
\usepackage{amsthm}
\usepackage{amsfonts}
\usepackage{amssymb}
\usepackage{graphics}
\usepackage{tabu}
\usepackage{breqn}
\usepackage[hang,flushmargin]{footmisc} 

\usepackage{amsmath,amsthm,amssymb,graphicx, multicol, array}
\usepackage{enumerate}
\usepackage{enumitem}
\usepackage{hyperref}
\usepackage{setspace}
\usepackage{xcolor}
\usepackage{currfile}

\usepackage{tabto}

\usepackage{pgfplots}
\pgfplotsset{width=10cm,compat=1.9}

\newtheorem{thm}{Theorem}[section]
\newtheorem{lem}{Lemma}[section]
\newtheorem{conj}{Conjecture}[section]

\newcommand{\N}{\mathbb{N}}

\newcommand{\C}{\mathbb{C}}

\usepackage{fancyhdr}
\title{Elementary Proof Of The Siegel-Walfisz Theorem}
\date{}
\author{N. A. Carella}

\begin{document}
\thispagestyle{empty}
\date{}

\maketitle

\textbf{\textit{Abstract}:} This note offers an elementary proof of the \textit{Siegel-Walfisz theorem} for primes in arithmetic progressions. \let\thefootnote\relax\footnote{ \today \date{} 
}

\section{Introduction} \label{s4420}
The techniques used to prove the leading prime number theorems are classified as either complex analytic methods or elementary methods. The complex analytic methods rely on the principles of real and complex analysis, while the elementary methods do not rely on these techniques. Let $n\geq 1$ be an integer, and let the vonMangoldt function be defined by
\begin{equation} \label{eq773.52}
\Lambda(n)=
\begin{cases}
\log p & \text{if } n=p^m,\\ 
0 & \text{if } n\ne p^m.\\ 
\end{cases}
\end{equation}
The above notation $p^m\geq 2$, with $ m \in \N$, denotes a prime power. The prime number theorem claims that
\begin{equation} \label{eq4422.032}
\psi(x)=\sum_{p\leq x}\Lambda(n)\sim x,
\end{equation}
and 
\begin{equation} \label{eq4422.030}
\pi(x)=\sum_{p\leq x}1\sim \frac{x}{\log x}.
\end{equation} 
The first complex analytic proofs were given by delaVallee Poussin, and Hadamard, independently, (listed in alphabetical order). These proofs are based on the zerofree region $\mathcal{H}(\beta)=\{s\in \C: \mathcal{R}e(s) \geq 1 -\beta\}$ of the zeta function $\zeta(s)=\sum_{n \geq 1}n^{-s}$, where $s \in \C$ is a complex number, and $\beta=o(1)$ is a decreasing function of $s$, see \cite[Theorem 6.9]{MV07}. The first elementary proofs of the prime number theorem proofs were given by Erdos in \cite{EP49} and Selberg in \cite{SA49}, independently, (listed in alphabetical order). Several refinements and simplified versions of the elementary proofs appear in \cite{DS70}, \cite[Theorem 3.10]{EL85}, et alii. Although the elementary proofs are visibly independent of the zeta function, there is an indirect link to the zeta function and complex analysis, confer \cite[Section 2]{GR09} for an explication. \\

Let $a<q$ be a pair of small fixed integers such that $\gcd(a,q)=1$. The prime number theorem in arithmetic progression claims that
\begin{equation} \label{eq4422.092}
\psi(x,q,a)=\sum_{\substack{n\leq x\\ n\equiv a \bmod q}}\Lambda(n) \sim x.
\end{equation}
and 
\begin{equation} \label{eq4422.090}
\pi(x,q,a)=\sum_{\substack{p\leq x\\ p\equiv a \bmod q}}1 \sim \frac{x}{\varphi(q)\log x}.
\end{equation}
The first complex analytic proof was given by Dirichlet, see \cite[p.\ 49]{NW00}. The literature has many versions of the original complex analytic proof, see \cite[Theorem 2]{SP73}, \cite[Theorem 7.9]{EL85}, et cetera. The first elementary proof  was given by Selberg in \cite{SA49B}. An extended version appears in \cite{GR92}. Surveys of the early and new developments of the prime number theorem appear in \cite{BD96}, \cite{NW00}, \cite{FJ15}, and similar references. \\

An effective form of the prime number theorem in arithmetic progression, best known as the Siegel-Walfisz theorem, has the form
\begin{equation} \label{eq4422.096}
\psi(x,q,a)=\frac{x}{\varphi(q)} \left (1 +O\left (\frac{\varphi(q)}{(\log x)^B}\right )\right )  ,
\end{equation}
and 
\begin{equation} \label{eq4422.094}
\pi(x,q,a)=\frac{x}{\varphi(q)\log x} \left (1 +O\left (\frac{\varphi(q)}{(\log x)^{B-1}}\right )\right )  ,
\end{equation}
where $B>C+1$ is an arbitrary constant and $q=O\left((\log x)^C\right )$, respectively. Versions of the complex analytic proof appears in \cite[Theorem 8.8]{EL85}, \cite{HM72}, 
\cite[Corollary 11.19]{MV07}, and a discussion in \cite[p.\ 405]{FI10}. The standard proofs are based on the zerofree regions $\mathcal{H}(\beta_{\chi})=\{s\in \C: \mathcal{R}e(s) \geq 1 -\beta_{\chi}\}$ of the Dirichlet $L$-function $L(s,\chi)=\sum_{n \geq 1}\chi(n)n^{-s}$, where $s \in \C$ is a complex number, $\chi$ is a character modulo $q$, and $\beta_{\chi}=o(1)$ is a decreasing function of $s$. \\

The literature does not seem to have an elementary proof of the Siegel-Walfisz theorem. This note supplies an elementary proof of the Siegel-Walfisz theorem. 
	
	\begin{thm}	\label{thm4422.002} Let $x \geq 1$ be a large number, and let $a<q\ll(\log x)^{C}$ be a pair of relatively prime integers, and $C\geq 0$. Then,
\begin{enumerate} [font=\normalfont, label=(\roman*)]
\item  $\displaystyle \sum_{\substack{n \leq x\\ n \equiv a \bmod q}} \Lambda(n)=	\frac{x}{ \varphi(q)} \left (1+O\left(\frac{\varphi(q)}{(\log x)^{B}}\right )\right ),$
\item  $\displaystyle \sum_{\substack{n \leq x\\ n \equiv a \bmod q}}1=	\frac{x}{ \varphi(q)\log x} \left (1+O\left(\frac{\varphi(q)}{(\log x)^{B-1}}\right )\right ),$
\end{enumerate}	
where $B>C+1$ is an arbitrary constant, and $B>C+1$.
\end{thm}
The asymptotic formula is uniformly valid and nontrivial for all $C<B-1$. The restriction $q\ll(\log x)^{C}$ is imposed by the form of the error term in \eqref{eq4422.096}. \\

The simple proof in Section \ref{s4422} is a direct derivation from the Prime Number Theorem. This elementary method do not require any information on the zerofree region of the Dirichlet $L$-function $L(s,\chi)=\sum_{n \geq 1} \chi(n)n^{-s}$, and it is much simpler than the complex analytic methods.  \\

\section{Prime Numbers Theorems In Arithmetic Progressions}  \label{s4422}

\begin{proof} (i) Replace the identity for the vonMangoldt function, in Lemma \ref{lem22.002}, and rearrange it into main term and error term:  
\begin{eqnarray}\label{eq4422.012}
\sum_{\substack{n \leq x\\ n \equiv a \bmod q}} \Lambda(n)
&=&-\sum_{\substack{n \leq x\\ n \equiv a \bmod q}}\left ( \sum_{\substack{d \mid n\\ d <\sqrt{n}}}\mu(d)\log d +\sum_{\substack{d \mid n\\ d \leq\sqrt{n}}}\mu(n/d)\log (n/d) \right ) \\
&=&-\sum_{\substack{n \leq x\\ n \equiv a \bmod q,}}\sum_{\substack{d \mid n\\ d <\sqrt{n}}}\mu(d)\log d -\sum_{\substack{n \leq x\\ n \equiv a \bmod q,}}\sum_{\substack{d \mid n\\ d \leq\sqrt{n}}}\mu(n/d)\log (n/d)  \nonumber\\
	&=&-\sum_{d <\sqrt{x}}\mu(d)\log (d) \sum_{\substack{n \leq x\\ d \mid n \\ n \equiv a \bmod q}} 1-\sum_{d \leq\sqrt{x}} \sum_{\substack{n \geq x\\ d \mid n\\ n \equiv a \bmod q}} \mu(n/d)\log (n/d)\nonumber \\
&=&M(x)\quad +\quad E(x)\nonumber.
\end{eqnarray}
Substituting the evaluation in Lemma \ref{lem8822.100} for the main term, and the upper bound in Lemma \ref{lem8822.200} for the error term, return	
\begin{eqnarray}\label{eq4422.016}
\sum_{\substack{n \leq x\\ n \equiv a \bmod q}} \Lambda(n)
&=&M(x)\quad +\quad E(x) \\
&=&  \frac{x}{ \varphi(q)} \left (1+O\left(\frac{ \varphi(q)}{(\log x)^B}
	\right )\right )+ O\left(\frac{x}{ (\log x)^{B}}
	\right )\nonumber\\
	&=&\frac{x}{ \varphi(q)} \left (1+O\left( \frac{\varphi(q)}{(\log x)^{B}}
	\right )\right )\nonumber,
\end{eqnarray}
where $B>C+1$ is an arbitrary constant, and $q\ll(\log x)^{C}$. Clearly, this is nontrivial count for $B>C+1$.  (ii) Use partial summation.

\end{proof}
\section{Some Elementary Identities}  \label{s22.00}
These identities are sort of pre-hyperbola method technique. Nevertheless, these identities offer the same efficiency as the general hyperbola method, see \cite[Theorem 3.17]{AP98}, and \cite[Equation 2.9]{MV07}, et cetera.

\begin{lem} \label{lem22.001} If $n \geq 1$ is an integer, $\mu(n)$ is the Mobius function, and $\Lambda(n)$ is the vonMangoldt function, then,
\begin{equation}\label{22.100}
\Lambda(n)=-\sum_{d\mid n}\mu(d)\log d .
\end{equation}	
\end{lem}

\begin{proof} Let $\log n= \sum_{d \mid n} \log(d)= \sum_{d \mid n} \Lambda(d)$, and use the Mobius inversion formula to compute its inverse.
\end{proof}

\begin{lem} \label{lem22.002} If $n \geq 1$ is an integer, $\mu(n)$ is the Mobius function, and $\Lambda(n)$ is the vonMangoldt function, then,
\begin{equation}\label{22.102}
\Lambda(n)=-\sum_{\substack{d\mid n\\d < \sqrt{x}}}\mu(d)\log d -\sum_{\substack{d\mid n\\d \leq \sqrt{x}}}\mu(n/d)\log(n/d).
\end{equation}	
\end{lem}
\begin{proof} Employ the basic fact that the set of divisors $d \mid n$ of any integer $n\geq 2$ can be partition as a union of two disjoint subsets

\begin{equation}\label{22.120}
\{ d\mid n\}=\{ d\mid n: d < \sqrt{n}\} \cup \{ d\mid n: d \geq \sqrt{n}\}.
\end{equation}
Lastly, substitute $ \{ d\mid n: d \geq \sqrt{n}\}= \{ n/d\mid n: d \leq \sqrt{n}\}$.
\end{proof}
Other complicated versions of identity \eqref{22.102} are available in the literature. The best known among these is the Vaughan identity, see \cite[Section 2.6]{HG07}, and the literature for more information.
\section{Some Elementary Finite Sums}  \label{s8822}

\begin{lem} \label{lem8822.100} Let $a<q$ be a pair of small fixed integers such that $\gcd(a,q)=1$. If $x \geq 1$ is a large number, and $\mu(n)$ is the Mobius function, then,
\begin{enumerate} [font=\normalfont, label=(\roman*)]
\item  $\displaystyle -\sum_{d <\sqrt{x}}\mu(d)\log (d) \sum_{\substack{n \leq x\\ d \mid n \\ n \equiv a \bmod q}} 1= \frac{x}{ \varphi(q)} \left (1+O\left(\frac{ \varphi(q)}{(\log x)^B}
	\right )\right ),$ \\

where $B>1$ is a constant, unconditionally.
\item  $\displaystyle -\sum_{d <\sqrt{x}}\mu(d)\log (d) \sum_{\substack{n \leq x\\ d \mid n \\ n \equiv a \bmod q}} 1= \frac{x}{ \varphi(q)} \left (1+O\left(\frac{ \varphi(q)\log x}{ x^{1/2}}
	\right )\right ),$ \\

conditional on the RH.
\end{enumerate}	

\end{lem}

\begin{proof} (i) The local constraints
\begin{multicols}{2}
\begin{enumerate} [font=\normalfont, label=(\roman*)]
\item  $\displaystyle a\ne 0$,
\item  $\displaystyle \gcd(a,q)=1,$
\item  $\displaystyle n \equiv a \bmod q,$
\item$\displaystyle n \equiv 0 \bmod d,$
\end{enumerate}
\end{multicols}
imply that there is a unique $b\ne0$, $\gcd(b,q)=1$ such that $n \equiv b \bmod dq$. Moreover, since $n=dm$, it follows that $m\equiv c \bmod q$ for some $c\ne0$. Accordingly, the main term is equivalent to
\begin{eqnarray}\label{eq2099.003}
-\sum_{d <\sqrt{x}}\mu(d)\log (d) \sum_{\substack{n \leq x\\ d \mid n \\ n \equiv a \bmod q}} 1
&=&	-\sum_{d <\sqrt{x}}\mu(d)\log (d)\sum_{\substack{n \leq x\\ n \equiv b \bmod dq}} 1\\
&=&	-\sum_{d <\sqrt{x}}\mu(d)\log (d)\sum_{\substack{m \leq x/d\\ m \equiv c \bmod q}} 1\nonumber.
\end{eqnarray}
Let $\{z\}\in (0,1)$ denotes the fractional part function. Then, the main term has the asymptotic expression
\begin{eqnarray}\label{eq2099.004}
-\sum_{d <\sqrt{x}}\mu(d)\log (d)\sum_{\substack{m \leq x/d\\ m \equiv c \bmod q}} 1&=&-\frac{1}{ \varphi(q)}\sum_{d <\sqrt{x}}\mu(d)\log (d)\left ( \frac{x}{d }-\left \{ \frac{x}{d }\right \}\right )\\
&=&-\frac{x}{ \varphi(q)} \sum_{d <\sqrt{x}}\frac{\mu(d)\log (d)}{d} -\sum_{d <\sqrt{x}}\mu(d)\log (d)\left \{ \frac{x}{d }\right \}\nonumber \\
&=&-\frac{x}{ \varphi(q)} \sum_{d <\sqrt{x}}\frac{\mu(d)\log (d)}{d} +O\left( \sqrt{x} (\log x)\right )\nonumber .
\end{eqnarray}
The partial sum in the third line of \eqref{eq2099.004} converges to a constant:
	\begin{equation}\label{eq2099.006}
	\sum_{n <\sqrt{x}}\frac{\mu(n)\log (n)}{n}=\sum_{n \geq 1}\frac{\mu(n)\log (n)}{n}-\sum_{d \geq \sqrt{x}}\frac{\mu(n)\log (n)}{n}=-1+
	O\left(\frac{ 1}{(\log x)^B}
	\right ),
	\end{equation}
where $B>1$ is a constant, see Lemma \ref{lem8822.300} for the detailed calculation. Merging \eqref{eq2099.004} and \eqref{eq2099.006} yield
\begin{eqnarray}\label{eq2099.008}
-\sum_{d <\sqrt{x}}\mu(d)\log (d)\sum_{\substack{m \leq x/d\\ m \equiv c \bmod q}} 1
&=&-\frac{x}{ \varphi(q)} \sum_{d <\sqrt{x}}\frac{\mu(d)\log (d)}{d} +O\left( \sqrt{x} (\log x)\right ) \\
&=&\frac{x}{ \varphi(q)} \left (1+O\left(\frac{  \varphi(q)}{(\log x)^B}
	\right )\right )\nonumber.
\end{eqnarray}
(ii) The partial sum in the third line of \eqref{eq2099.004} converges to a constant:
	\begin{equation}\label{eq2099.010}
	\sum_{n <\sqrt{x}}\frac{\mu(n)\log (n)}{n}=\sum_{n \geq 1}\frac{\mu(n)\log (n)}{n}-\sum_{d \geq \sqrt{x}}\frac{\mu(n)\log (n)}{n}=-1+
	O\left(\frac{ (\log x)^2}{ x^{1/2}}
	\right ),
	\end{equation}
Merging \eqref{eq2099.004} and \eqref{eq2099.008} to complete the proof.
\end{proof}

\begin{lem} \label{lem8822.200} Let $a<q$ be a pair of small fixed integers such that $\gcd(a,q)=1$. If $x \geq 1$ is a large number, and $\mu(n)$ is the Mobius function, then,
\begin{enumerate}[font=\normalfont, label=(\roman*)]
\item $\displaystyle-\sum_{d \leq\sqrt{x}} \sum_{\substack{n \leq x\\ d \mid n\\ n \equiv a \bmod q}} \mu(n/d)\log (n/d)= O\left(\frac{x}{ (\log x)^{B}}
	\right ),$\\

where $B>1$ is a constant, unconditionally.
    
 \item $\displaystyle -\sum_{d \leq\sqrt{x}} \sum_{\substack{n \leq x\\ d \mid n\\ n \equiv a \bmod q}} \mu(n/d)\log (n/d)= O\left(x^{1/2+\varepsilon}\right ),$\\

conditional on the RH for any small number $\varepsilon>0$.
 \end{enumerate}
\end{lem}

\begin{proof}(i) Let $n=dm\leq x$. The local constraints
\begin{multicols}{2}
\begin{enumerate} [font=\normalfont, label=(\roman*)]
\item  $\displaystyle a\ne 0$,
\item  $\displaystyle \gcd(a,q)=1,$
\item  $\displaystyle n \equiv a \bmod q,$
\item$\displaystyle n \equiv 0 \bmod d,$
\end{enumerate}
\end{multicols}
imply that there is a unique $b\ne0$, $\gcd(b,q)=1$ such that $n \equiv b \bmod dq$. Moreover, (iii) and (iv) imply that $m\equiv c \bmod q$ for some $c\ne0$. Thus, the error term has the asymptotic
	\begin{eqnarray}\label{eq2099.008}
	-\sum_{d \leq\sqrt{x}} \sum_{\substack{n \leq x\\ d \mid n\\ n \equiv a \bmod q}} \mu(n/d)\log (n/d)
	&=&-\sum_{d \leq\sqrt{x}} \sum_{\substack{n \leq x\\ n \equiv b \bmod dq}} \mu(n/d)\log (n/d)\\
	&=&-\sum_{d \leq\sqrt{x}} \sum_{\substack{m \leq x/d\\ m \equiv c \bmod q}} \mu(m)\log (m)\nonumber. 
	\end{eqnarray}
Applying Theorem \ref{thm8822.459} to the last finite sum in \eqref{eq2099.008} yield the upper bound
\begin{eqnarray}\label{eq2099.010}
-\sum_{d \leq\sqrt{x}} \sum_{\substack{m \leq x/d\\ m \equiv c \bmod q}} \mu(m)\log (m)
&=&O\left(\frac{x}{(\log x)^{D}}
\sum_{d \leq\sqrt{x}} \frac{1}{d}\right )\\
&=&O\left(\frac{x}{ (\log x)^{D-1}}	\right )\nonumber, 
\end{eqnarray}
where $a<q$ be a pair of small fixed integers such that $\gcd(a,q)=1$, and $B>D+1$ is an arbitrary constant.	
\end{proof}

\section{Some Elementary Foundation}  \label{s8822}
The relations
\begin{equation}\label{eq8822.300}
\sum_{n \geq1} \frac{ \mu(n)}{n}=0\qquad \text{ and } \qquad \sum_{n \geq1} \frac{ \mu(n)\log (n)}{n}=-1
\end{equation}
are known to be equivalent to the prime number theorem, as claimed in \cite[p.\ 78]{EL85}, \cite[p.\ 285]{NW00}, et alii. An elementary proof of the second one is given \textit{in situ}.
\begin{lem} \label{lem8822.300} Let $\mu(n)$ be the Mobius function, then,
\begin{equation}\label{eq8822.302}
\frac{\zeta^{\prime}(1)}{\zeta(1)^2}=\sum_{n \geq1} \frac{ \mu(n)\log (n)}{n}=-1.
\end{equation}	
\end{lem}

\begin{proof} Let $1/\zeta(s)=\sum_{n \geq 1}\mu(n)n^{-s}$. The derivative at $s=1$ can be computed using the Taylor series 
\begin{equation}\label{eq8822.303}
\zeta(s)= \frac{ 1}{s-1}+c_0+c_1(s-1)+c_2(s-1)^2+\cdots,
\end{equation}
of the zeta function at $s=1$, see \cite[Equation 25.2.4]{DLMF}, and a derivation in \cite{BF88}. Specifically,
\begin{equation}\label{eq8822.305}
\frac{\zeta^{\prime}(s)}{\zeta(s)^2}= \frac{\frac{- 1}{(s-1)^2}+c_1+c_2(s-1)+\cdots}{\left (\frac{ 1}{s-1}+c_0+c_1(s-1)+c_2(s-1)^2+\cdots\right )^2}.
\end{equation}
Taking the limit at $s=1$ yields
\begin{equation}\label{eq8822.307}
\lim_{s\to 1} \frac{(s-1)^2}{(s-1)^2}\frac{\zeta^{\prime}(s)}{\zeta(s)^2}= -1.
\end{equation}
\end{proof}
There are other means of evaluating the derivative, see \cite[Exersice 16, p.\ 184]{MV07}. \\

The basic estimates for finite sums of the forms $\sum_{n \leq x} \mu(n)\log (n)$ used in Lemma \ref{lem8822.200} are derive from the standard results given below.
\begin{thm} \label{thm8822.459} Let $a<q$ be a pair of small fixed integers such that $\gcd(a,q)=1$, and  let $\mu(n)$ be the Mobius function. Then, 
\begin{enumerate}[font=\normalfont, label=(\roman*)]
\item $\displaystyle \sum_{\substack{n \leq x\\ n \equiv a \bmod q}} \mu(n)=O\left(\frac{x}{\log^D x} \right ),$ \tabto{6cm} unconditionally, for any constant $D>0$,
    
 \item $\displaystyle \sum_{\substack{n \leq x\\ n \equiv a \bmod q}} \mu(n)=O\left(x^{1/2+\varepsilon} \right ),$ \tabto{6cm} conditional on the RH, for any constant $\varepsilon>0$.
 \end{enumerate}
\end{thm}

\begin{proof}Confer \cite[p.\ 182]{MV07}, and similar references.
\end{proof}
\section{ Conditional Results And Conjecture}  \label{s2030}
Under the generalized Riemann hypothesis, the prime number theorem for the number of primes $p\leq x$ in the arithmetic progression 

\begin{equation} \label{eq2030.05}
\mathcal{A}(q,a)=\{p=qn+a: n \geq 1\},
\end{equation}
specifies the asymptotic formulas
\begin{equation} \label{eq2030.00}
\psi(x,q,a)=\frac{x}{\varphi(q)}  +O\left (x^{1/2}(\log x)^2\right ),
\end{equation}
and
\begin{equation} \label{eq2030.00}
\pi(x,q,a)=\frac{x}{\varphi(q)\log x}  +O\left (x^{1/2}\log x\right ).
\end{equation}
This is uniformly nontrivial for $q\ll x^{1/2}/\log(x)^C$. This restriction is imposed by the form of the error term in \eqref{eq2030.00}, more information appear in \cite[p.\ 419]{IK04}. \\

Under the Riemann hypothesis, both of the expressions \eqref{eq2030.05} and \eqref{eq2030.05} can be derived using the same elementary techniques demonstrated in Sections \ref{s4422} to \ref{s8822}, and it is independent of the zerofree regions of the corresponding $L$-function.  The precise statements are as follow.

\begin{thm} \label{thm2030.200} Assume the RH. Let $1\leq a<q\ll x^{1/2}/\log(x)^C$ be a pair of small fixed integers such that $\gcd(a,q)=1$, and $C\geq 2$ is a constant. Then, 
\begin{enumerate}[font=\normalfont, label=(\roman*)]
\item $\displaystyle \psi(x,q,a)=\frac{x}{\varphi(q)}  +O\left (x^{1/2}(\log x)^2\right ),$
    
 \item $\displaystyle \pi(x,q,a)=\frac{x}{\varphi(q)\log x}  +O\left (x^{1/2}\log x\right ).$
 \end{enumerate}
\end{thm}

\begin{proof} (i) Proceed as in \eqref{eq4422.012}, and \eqref{eq4422.016}. Substituting the conditional evaluation in Lemma \ref{lem8822.100} for the main term, and the conditional upper bound in Lemma \ref{lem8822.200} for the error term, return	
\begin{eqnarray}\label{eq2030.016}
\sum_{\substack{n \leq x\\ n \equiv a \bmod q}} \Lambda(n)
&=&M(x)\quad +\quad E(x) \\
&=&  \frac{x}{ \varphi(q)} \left (1+O\left(\frac{ \varphi(q)\log x}{x^{1/2}}
	\right )\right )+ O\left(x^{1/2} (\log x)^{2}
	\right )\nonumber\\
	&=&\frac{x}{ \varphi(q)} \left (1+O\left(\frac{ \varphi(q)(\log x)^{2}}{x^{1/2}}
	\right )\right )\nonumber.
\end{eqnarray}
(ii) Use partial summation.
\end{proof}

The Montgomery conjecture, stated below, does not seem to have similar elementary proof derived directly from the RH.
\begin{conj}\label{conj2099.35} {\normalfont (\cite[Conjecture 13.9]{MV07}} Let $a<q$ be integers, $\gcd(a,q)=1$, and $q \leq x$. Then,
\begin{enumerate} [font=\normalfont, label=(\roman*)]
\item  $\displaystyle \psi(x,q,a)=\frac{x}{\varphi(q)}+O\left ( \frac{x^{1/2+\varepsilon}}{q^{1/2}}\right ).$
\item  $\displaystyle \pi(x,q,a)=\frac{x}{\varphi(q)\log x}+O\left ( \frac{x^{1/2+\varepsilon}}{q^{1/2}}\right ).$
\end{enumerate}	

\end{conj}
The Montgomery conjecture for primes in arithmetic progression, recorded above, extends the range of moduli to nearly all $q<x$, but not all $q\leq x$. The advanced theory for equidistributions and oscillations in the prime number theorems was pioneered and developed in a series of papers, \cite{MH88}, \cite{FG91}, et cetera.


\currfilename.\\
\end{document}